\documentclass{article}

\usepackage{cite}
\usepackage{amssymb}
\usepackage{amsmath}
\usepackage{amsfonts}
\usepackage{amsthm}
\usepackage{fancyhdr}
\usepackage[explicit]{titlesec}
\usepackage{titletoc}
\usepackage{booktabs}
\usepackage[all]{xy}
\usepackage[colorlinks=true,linkcolor=blue]{hyperref}

\usepackage{cancel}

\usepackage[mathscr]{euscript}

\usepackage{url}

\usepackage{hyperref}

\newtheorem{proposition}{Proposition}[section]

\newtheorem*{theorem*}{Theorem}
\newtheorem{theorem}{Theorem}[section]
\newtheorem{corollary}{Corollary}[section]
\newtheorem{definition}{Definition}[section]
\newtheorem{example}{Example}[section]
\newtheorem{lemma}{Lemma}[section]

\newtheorem{remark}{Remark}[section]
\newtheorem{problem}{Problem}[section]

\numberwithin{equation}{section}

\everymath{\displaystyle}

\begin{document}

\title{A quotient-like construction involving elementary submodels.}

\author{Peter Burton}

\date{\today}

\maketitle

\begin{abstract} 
This article is an investigation of a method of deriving a topology from a space and an elementary submodel containing it. We first define and give the basic properties of this construction, known as $X/M$. In the next section, we construct some examples and analyse the topological relationship between $X$ and $X/M$. In the final section, we apply $X/M$ to get novel results about Lindel\"{o}f spaces, giving partial answers to a question of F.D. Tall and another question of Tall and M. Scheepers. \end{abstract}

\begin{section}{Definitions and preliminaries.} We assume all spaces are $T_{3\frac{1}{2}}$ - since we are primarily interested in strengthenings of the Lindel\"{o}f property, this causes no serious loss of generality. Given an elementary submodel $M \prec H_\lambda$ and a space $X$ with $X \in M$, we will refer to the classical notion of `subspace with respect to $M$' as $X_M$. Recall that this is the topology on $X \cap M$ generated by the family $(U \cap M: U \in M$ is open in $X)$ [7]. Our object of study is $X/M$, which can be seen as analogous to $X_M$ where `subspace' is replaced by `quotient'. The construction of $X/M$ was introduced independently by I. Bandlow in [1] and A. Dow in [2], and discussed further by T. Eisworth in[3]. There are two related ways to define $X/M$; we give the more elementary method first.

\begin{definition}\label{def1} Define an equivalence relation on $X$ by letting $x_0 \sim x_1$ if and only if $f(x_0) = f(x_1)$ for all continuous $f:X \to \mathbb{R}$ such that $f \in M$. Let $X/M$ be the resulting quotient set and write $\pi:X \twoheadrightarrow X/M$ for the projection. We topologize $X/M$ by taking a base to be all sets of the form $\pi(U)$, where $U \in M$ is a cozero set in $X$. \end{definition}

\begin{definition}\label{def2} Since $X$ is $T_{3\frac{1}{2}}$, we can consider the natural embedding $e: X \hookrightarrow [0,1]^{C^*(X)}$ given by $e(x) = (f(x))_{f \in C^*(X)}$. There is also a natural map $F$ from $[0,1]^{C^*(X)}$ to $[0,1]^{C^*(X) \cap M}$, given by $(x_f)_{f \in C^*(X)} \mapsto (x_f)_{f \in C^*(X) \cap M}$. Define $X/M$ to be $F \circ e(X)$. \end{definition}

The equivalence of these formulations is proven in [3], with the correspondence given by sending the equivalence class $[x] \in X/M$ of a point in $x \in X$ to $F \circ e(x)$. We will use them interchangably. We now summarize the relevant basic properties of $X/M$. 

\begin{lemma}\label{lem3} \begin{enumerate} \item $X/M$ is a continuous image of $X$,
\item $x_0 \sim x_1$ exactly when $x_0 \in U$ if and only if $x_1 \in U$ for each cozero $U \in M$,
\item $X/M$ is a $T_{3\frac{1}{2}}$ space,
\item $X_M$ embeds as a dense subspace of $X/M$; we will identify $X_M$ with the copy $\{ \pi(x):x \in X \cap M\} \subseteq X/M$, and
\item $[x] = \bigcap \{Z: x \in Z \in M$ and $Z$ is a zero set $\} = \bigcap \{\overline{U}: x \in U \in M$ and $U$ is a cozero set $\}$. \end{enumerate}\end{lemma}

Here, $(4)$ and $(5)$ are from [3]. $(1)$ is trivial, while $(2)$ is immediate from Definition \ref{def1}. $(3)$ follows from the fact that Definition \ref{def2} embeds $X/M$ in the compact space $[0,1]^{C^*(X) \cap M}$. We also have the following lemmas, which will prove useful in our later applications.

\begin{lemma}\label{lem9} Suppose $U \in M$ is a cozero set. Then, $\pi(U)$ is an open set in $X/M$ and $\pi^{-1}\pi(U) = U$. \end{lemma}

\begin{proof} The first conclusion is obvious from \ref{def1}, so suppose $\pi(x) \in \pi(U)$. This means that $x \sim x_0$ for some $x_0 \in U$. Thus, $f(x) = f(x_0)$ for every continuous $f: X \to \mathbb{R}$ with $f \in M$. But, since $U \in M$, there is a continuous $f_0 \in M$ with $U = f_0^{-1}(0,1]$. Hence, $f_0(x) = f_0(x_0) \in (0,1]$, which implies that $x \in U$. \end{proof}

\begin{lemma}\label{lem10} Suppose $U \in M$ is a cozero set. Then for any $M$-equivalence class $[x]$, if $[x] \cap U \neq \emptyset$ then $[x] \subseteq U$. \end{lemma}

\begin{proof} Lemma \ref{lem3}(2) implies that if $x \in U$ and $x_1 \sim x$, then $x_1 \in U$.\end{proof}

We now provide a new characterization of $X/M$ as `universal' among images of $X$. Fix a space $X$ and a model $M$ and say that a pair $(Y,F)$, $F:X \to Y$ is \textbf{described by} $\mathbf{M}$ if $Y$ has a base $\mathcal{B}$ such that $F^{-1}(B)$ is a cozero set in $M$ for every $B \in \mathcal{B}$. Note that if $(Y,F)$ is described by $M$, then $F$ is continuous. Also note that we do not require $Y$ or $F$ to be members of $M$. This ensures that $(X/M, \pi)$ is described by $M$, along with any pair $(Y, F)$ for which $F \in M$ and $Y$ has a base included in $M$.

\begin{theorem}\label{thm1} Suppose $X \in M$ and let $(Y,F)$ be described by $M$. Then there is a continuous surjection $f: X/M \to Y$ such that the following diagram commutes. 

\[
\xymatrix{
X \ar[dd]^{\pi} \ar[ddrr]^{F} \\
\\
X/M \ar[rr]_f && Y} \] Moreover, $(X/M, \pi)$ is, up to a homeomorphism, the unique pair described by $M$ for which this conclusion holds.
 \end{theorem}

\begin{proof} Given $X, F$ and $Y$ we seek to define $f: X/M \to Y$ in the natural way, that is $f([x]) = F(x)$. Suppose that $F(x_0) \neq F(x_1)$ for points $x_0$ and $x_1 \in X$. We must check that $[x_0] \neq [x_1]$. There are disjoint $U_0, U_1 \in \mathcal{B}$ such that $[x_0] \in F(U_0)$ and $[x_1] \in F(U_1)$. $V_0 = F^{-1}(U_0)$ and $V_1 = F^{-1}(U_1)$ are disjoint cozero sets in $X$ with $x_0 \in V_0$ and $x_1 \in V_1$, and moreover both $V_0$ and $V_1$ are elements of $M$. Then Lemma \ref{lem3}(2) guarantees that $x_0$ is not equivalent to $x_1$, and our map $f$ is well defined. It is clear that $f$ is continuous and that $f \circ \pi = F$.\\

Now, suppose $(Z, \tau)$ is another pair described by $M$ with the same property as the theorem asserts for $(X/M,\pi)$. That is to say, for any other $(Y,F)$ described by $M$ there is a continuous surjection $f:Z \to Y$ such that $F = f \circ \tau$.\\

In particular, we have a continuous $f: Z \twoheadrightarrow X/M$ such that $f \circ \tau = \pi$, and a continuous $g:X/M \twoheadrightarrow Z$ with $g \circ \pi = \tau$. Then for any $\tau(x) \in Z$,
\[g\bigl(f(\tau(x))\bigr) = g(\pi(x)) = \tau(x),\]
while for $\pi(x) \in X/M$,
\[f\bigl(g(\pi(x))\bigr) = f(\tau(x)) = \pi(x).\]
This shows that $f$ and $g$ are inverse mappings; clearly $f$ is a homeomorphism from $Z$ to $X/M$.

 \end{proof}

We have three useful corollaries, the first of which can be found in ~\cite{Eisworth2006}.

\begin{corollary}\label{cor1} If $M_0 \prec M_1$ are elementary submodels containing $X$, then $X/M_1$ maps continuously onto $X/M_0$. \end{corollary}

\begin{corollary}\label{cor2} If $g \in C(X) \cap M$, then $g$ factors through $\pi$, that is, there is a continuous map $f \in C(X/M)$ with $g = f \circ \pi$. \end{corollary}\end{section}

\begin{corollary}\label{cor9} Suppose $(Z,\tau)$ is a pair described by $M$ and that $Z$ is a continuous preimage of $X/M$. Then $Z$ is homeomorphic to $X/M$.\end{corollary}

\begin{proof} Let $h$ map $Z$ onto $X/M$; if $(Y,F)$ is another pair described by $M$ then there is a continuous surjection $f:X/M \to Y$. So, $h \circ f$ is a continuous surjection $Z \to Y$. Then $Z$ is homeomorphic to $X/M$ by the uniqueness asserted in Theorem \ref{thm1},\end{proof}

\begin{section}{Examples and analysis.}

We now investigate the topological relationship between $X$ and $X/M$. Our first result characterizes when we have $X/M$ homeomorphic to $X$. 

\begin{theorem} $\pi: X \to X/M$ is a homeomorphism if and only if $M$ includes a base for $X$.\end{theorem}

\begin{proof} Suppose first that there is a base $\mathcal{B}$ for $X$ such that $\mathcal{B} \subseteq M$; we may assume that $\mathcal{B}$ consists of cozero sets. Then the pair $(X, \mathrm{id}_X)$ will satisfy the hypotheses on $(Y,F)$ in Theorem \ref{thm1}. It is easy to see that the corresponding $f:X/M \to X$ is a continuous inverse for $\pi: X \to X/M$. \\

Suppose now that $\pi:X \to X/M$ is a homeomorphism. We know that $X/M$ has a base $\mathcal{B} = \bigl \{ \pi(U): U \in M$ is a cozero set in $X \bigr \}$. By Lemma \ref{lem9}, $\pi^{-1}\pi(U) = U$ for every $\pi(U) \in \mathcal{B}$. But $\pi^{-1}$ is a homeomorphism, so $\bigl\{\pi^{-1}\pi(U): U \in M$ is a cozero set$\bigr \} = \{ U: U \in M$ is a cozero set$\} \subseteq M$ should be a base for $X$. \end{proof}

\begin{corollary}\label{cor8} If $M$ is any elementary submodel and $X$ is a space with a countable base $\mathcal{B}$, we may assume that $\mathcal{B} \subseteq M$ and hence $X/M = X$. In particular, $\mathbb{R}/M = \mathbb{R}$ for any elementary submodel $M$. \end{corollary}

\begin{remark} Observe that if $X_M = X$ then $M$ must include a base for $X$. Hence, any of the conditions for $X_M = X$ given in [8], [15] or [17] are sufficient to guarantee that $X/M = X$. We note an example from [16] that is particularly interesting in view of our later considerations about indestructibly Lindel\"{o}f spaces.
\begin{itemize} \item If $X_M$ is a continuous image of $\{0,1\}^\kappa$ for some $\kappa$, $w(X_M)$ is a member of $M$ and less than the first strongly inaccessible cardinal, then $X_M = X$.\end{itemize}\end{remark}

\begin{remark} There are, however, simple cases where $X/M = X$ and $X_M$ is a proper subset of $X$. $\mathbb{R}/M = \mathbb{R}$ for any elementary submodel $M$ by Corollary \ref{cor8}, but $\mathbb{R}_M$ is countable when $M$ is countable.\end{remark}

\begin{subsection}{Cardinal invariants of $X/M$}

Naturally, all relationships among cardinal invariants preserved by a continuous map hold for $X$ and $X/M$; see [5] for a discussion of these. The most interesting cardinal inequalities for $X/M$ that we have found are provided by the following example. Recall that the character of a point $x \in X$ is denoted $\chi(x,X)$ and defined as the least cardinality of a local base at $x$. The character $\chi(X)$ is $\max\bigl(\sup_{x \in X}\chi(x,X), \aleph_0 \bigr)$. Similarly, the pseudocharacter at $x \in X$ is denoted $\psi(x,X)$ and defined as $\min (|\mathcal{U}|: \mathcal{U}$ is a family of open sets and $\bigcap \mathcal{U} = \{x\})$. The pseudocharacter $\psi(X)$ is $\max\bigl(\sup_{x \in X} \psi(x,X), \aleph_0\bigr)$. Clearly, $\psi(X) \leq \chi(X)$ for $T_1$ spaces $X$. The idea behind this example is that pointwise invariants such as $\psi$ and $\chi$ need not be preserved by $\pi:X \to X/M$.

\begin{example}\label{ex1} Consider $D(2^{\beth_{\omega_1}})$, the discrete space of size $2^{\beth_{\omega_1}}$. Then there is a model $M$ of size $\beth_{\omega_1}$ such that $\chi\bigl(D(2^{\beth_{\omega_1}})\bigr) < \chi\bigl(D(2^{\beth_{\omega_1}})/M\bigr)$ and $\psi\bigl(D(2^{\beth_{\omega_1}})\bigr) < \psi\bigl(D(2^{\beth_{\omega_1}})/M\bigr)$. \end{example}

Note that $\chi\bigl(D(2^{\beth_{\omega_1}})\bigr) = \psi\bigl(D(2^{\beth_{\omega_1}})\bigr) = \aleph_0$.  For a model $M$ of size $\beth_{\omega_1}$, $D(2^{\beth_{\omega_1}})/M$ has a base of size $\beth_{\omega_1}$, so if $\chi\bigl(D(2^{\beth_{\omega_1}})/M\bigr) = \aleph_0$ or $\psi\bigl(D(2^{\beth_{\omega_1}})\bigr) = \aleph_0$ then we must have $|D(2^{\beth_{\omega_1}})/M| = \beth_{\omega_1}^{\aleph_0} = \beth_{\omega_1}$. Hence, it suffices to find a model $M$ of size $\beth_{\omega_1}$ such that $\beth_{\omega_1}< |D(2^{\beth_{\omega_1}})/M|$.\\

Regard $D(2^{\beth_{\omega_1}})$ as the binary splitting tree $2^{\beth_{\omega_1}}$ with the box topology. For each $\alpha < \beth_{\omega_1}$, there are $2^\alpha \leq \beth_{\omega_1}$ branches through the first $\beth_{\omega_1}$ levels of the tree $D(2^{\beth_{\omega_1}})$. List these branches as $(B^\alpha_\beta)_{\beta < 2^\alpha}$. For each $B^\alpha_\beta$, define a function $f^\alpha_\beta:D(2^{\beth_{\omega_1}}) \to \mathbb{R}$ by $f^\alpha_\beta(x) = 0$ if $x \upharpoonright \alpha = B^\alpha_\beta$ and $f(x) = 1$ otherwise. Note $f^\alpha_\beta$ is continuous. Furthermore, we may consider a model $M$ with size $\beth_{\omega_1}$ such that $\bigl( \bigcup_{\alpha < \beth_{\omega_1}, \beta < 2^\alpha} f^\alpha_\beta\bigr) \subseteq M$. Two points in $D(2^{\beth_{\omega_1}})$ are $M$-equivalent exactly when they define the same branch through the tree, hence $D(2^{\beth_{\omega_1}})/M = D(2^{\beth_{\omega_1}})$ as sets (note they will not be homeomorphic, since $w(D(2^{\beth_{\omega_1}})/M) = \kappa$). This shows that $\beth_{\omega_1} < |D(2^{\beth_{\omega_1}})/M|$ for $|M| = \beth_{\omega_1}$, so $D(2^{\beth_{\omega_1}})/M$ cannot have countable pseudocharacter.

\begin{subsection}{Compactness and connectedness of $X/M$.}

Junqueira has shown in [6] that if $X_M$ is compact, then $X$ itself must be compact. Moreover, Theorem 1.1 in [8] gives several cases in which $X_M = X$ for compact $X_M$. We will see that the situation is entirely different for $X/M$.

\begin{theorem}\label{thm2} Suppose $\gamma$ is an ordinal and $\gamma \in M$. Then, for any $\kappa, \lambda$ and a family $\mathcal{M}$ of models closed and cofinal in $\{M \prec H_\lambda: |M| = \kappa \}$, we have that $X/M = (M \cap \gamma)+1$ for any $M \in \mathcal{M}$. \end{theorem}

\begin{proof} The key point in this argument is that $M \cap \gamma$ is an initial segment of $\gamma$. Also assume that $f_\alpha, g_\alpha \in M$ whenever $\alpha \in M$, where $f_\alpha: \gamma \to \mathbb{R}$ is given by $f_\alpha(\beta) = \begin{cases} 0 &\mbox{if } \beta \leq \alpha \\ 
1 & \mbox{if } \alpha < \beta. \end{cases}$ and $g_\alpha$ is given by $\begin{cases} 1 &\mbox{if } \beta \leq \alpha \\ 
0 & \mbox{if } \alpha < \beta. \end{cases}$. Finally, assume that for any cozero set $U \in M$, the ordinal $\sup U$ is a member of $M$. Consider $\gamma /M$, and claim that for any $\alpha < \gamma \cap M$, the class $[\alpha]$ is a singleton. Suppose to the contrary that $\alpha_0 \sim \alpha_1$, where we may assume that $\alpha_0 \leq \alpha_1$. Then we should have $f(\alpha_0) = f(\alpha_1)$ for all $f \in M$, which contradicts our assumption that $f_{\alpha_0} \in M$. \\

We have assumed that the operation $U \mapsto \sup U$ is in $M$ for $U$ cozero in $M$. Hence, we must have $U \subseteq \gamma \cap M$ for every $U \in M$. But, this implies that for all $U \in M$ and all $\alpha > M \cap \gamma$, $\alpha \notin U$. Hence, if $\alpha_0, \alpha_1 > \gamma \cap M$, $\alpha_0$ is vacuously equivalent to $\alpha_1$. Therefore, $\pi\bigl( (\gamma \cap M, \gamma) \bigr)$ is a single equivalence class in $\gamma /M$.\\

We can now unambigiously define a wellordering on $\gamma /M$ by $[\alpha_0] \leq [\alpha_1]$ if and only if $\alpha_0 \leq \alpha_1$. This is a wellordering since $\leq$ is a wellordering. We will show that the rays form a basis for $\gamma/M$. Suppose that $\alpha < \gamma \cap M$. Then, $\pi^{-1}\bigl( (0, [\alpha]) \bigr) = (0, \alpha)$ and $\pi^{-1}\bigl( ([\alpha], [\gamma \cap M]) \bigr) = (\alpha, \gamma)$. Both $(0,\alpha)$ and $(\alpha, \gamma)$ are cozero members of $M$ since we assumed that $f_\alpha$ and $g_\alpha$ are in $M$. Thus, $(0, [\alpha])$ and $([\alpha], [\gamma \cap M])$ are open in $\gamma /M$. The case when $\alpha \geq \gamma \cap M$ is obvious, so we have shown that rays are open in $\gamma /M$. Moreover, we have shown that the image of a ray in $\gamma$ is an open ray in $\gamma/M$, so the open rays of $\gamma/M$ form a basis for $\gamma/M$. \\

To conclude, we have shown that $\gamma /M$ is a wellordered space with the order topology, which consists of singleton classes $[\alpha]$ for $\alpha < \gamma \cap M$ and a unique class $[\gamma \cap M]$ for all $\alpha \geq \gamma \cap M$. It is now clear that $\gamma/M$ is in fact isomorphic to $\gamma \cap M +1$, where we identify the element $\gamma \cap M$ with $[\gamma \cap M]$. Corollary \ref{cor1} then gives that any model $M'$ with $M \prec M'$ will have $\gamma /(M')$ homeomorphic to $(M' \cap \gamma )+1$\end{proof}

\begin{remark}\label{rem1} This shows that for any ordinal $\gamma$ and a club of models $\mathcal{M}$, $\gamma /M$ is compact for each $M \in \mathcal{M}$. We refer to a club in the sense of a family which is cofinal and closed under increasing $\kappa$-sequences with respect to inclusion in $(M \prec H_\lambda: |M| = \kappa)$ for some fixed $\kappa$. Recall that the Lindel\"{o}f number $L(X)$ of a space $X$ is defined to be the least cardinal $\kappa$ such that every open cover of $X$ has a subcover of size $\kappa$ - so, for example, the Lindel\"{o}f property is equivalent to the statement $L(X) = \aleph_0$. Since for a regular ordinal $\gamma$ we have $L(\gamma) = \gamma$, Theorem \ref{thm2} demonstrates that there are spaces $X$ of arbitrarily large Lindel\"{o}f number such that $X/M$ is compact. \end{remark}

We now generalize an unpublished result of Todd Eisworth to get a contrasting statement.

\begin{proposition} The following are equivalent for a Tychonoff space $X$. \begin{enumerate} \item $X$ is pseudocompact, \item $X/M$ is pseudocompact whenever $M$ is an elementary submodel containing $X$, \item there is an elementary submodel $M$ containing $X$ for which $X/M$ is pseudocompact. \end{enumerate} \end{proposition}

\begin{proof} $(1)$ implies $(2)$ as continuous images of pseudocompact spaces are pseudocompact. $(2)$ implies $(3)$ trivially. To see that $(3)$ implies $(1)$, assume $X$ is not pseudocompact. Then there is a continuous function from $X$ to $\mathbb{R}$ with unbounded range. If $M$ is any submodel containing $X$, then $M$ is going to contain such a continuous function by elementarity. But this function then induces a continuous map from $X/M$ to $\mathbb{R}$ and it's easy to check the induced map has unbounded range as well. Thus $X/M$ is not pseudocompact. \end{proof}

\begin{lemma}[20] A space is realcompact and pseudocompact if and only if it is compact \end{lemma}

We now have an interesting counterpoint to Remark \ref{rem1}.

\begin{theorem} Let $X$ be realcompact. Then $X$ is compact if and only if $X/M$ is pseudocompact for some elementary submodel $M$ containing $X$. \end{theorem}

Unlike compactness, the connectedness of $X$ is determined very easily from the connectedness of $X/M$.

\begin{theorem} The following are equivalent for a Tychonoff space $X$. \begin{enumerate} \item $X$ is connected, \item $X/M$ is connected whenever $M$ is an elementary submodel containing $X$, \item there is an elementary submodel $M$ containing $X$ for which $X/M$ is connected. \end{enumerate} \end{theorem}

\begin{proof} $(1)$ implies $(2)$ since continuous images of connected spaces are connected. $(2)$ implies $(3)$ is trivial. To see $(3)$ implies $(1)$, suppose that $X$ is disconnected. Then there is a continuous $g: X \to \mathbb{R}$ such that $g(X) = \{0,1\}$. Hence, $g \in M$ for some $M$. But then Corollary \ref{cor2} gives a continuous $g: X/M \to \mathbb{R}$ such that $g(X/M) = \{0,1\}$, which contradicts the connectedness of $X/M$. \end{proof} \end{subsection}

\end{subsection}

\begin{subsection}{X/M as a subspace.}

We will say that $Y \subseteq X$ is a \textbf{weak subspace} if $Y$ has a coarser topology than the subspace topology induced from $X$.

\begin{lemma}\label{lem5} For any elementary submodel $M$, $X/M$ is homeomorphic to a weak subspace of $X$. \end{lemma}

\begin{proof} Construct a subset $Y$ of $X$ by choosing a single point from each $M$-equivalence class. Endow $Y$ with the topology generated by $(U \cap Y: U \in M$ is a cozero set$)$. Clearly, $Y$ is a weak subspace of $X$. If $\pi: X \twoheadrightarrow X/M$ is the projection, then $\pi \upharpoonright Y: Y \twoheadrightarrow X/M$ is a continuous mapping from $Y$ onto $X/M$. Consider the map $\tau: X/M \to Y$ given by assigning to $\pi(x) \in X/M$ the unique point in $[x] \cap Y$, where $[x]$ is the $M$-equivalence class of $x$ viewed as a subset of $X$. For a basic open set $U \cap Y$, $U \in M$ we must have that $\pi^{-1}\tau^{-1}(U \cap Y) = U$. Hence, $(Y, \tau \circ \pi)$ is described by $M$ in the sense of Theorem \ref{thm1} and $Y$ is homeomorphic to $X/M$ by Corollary \ref{cor9}. \end{proof}

\end{subsection}

\end{section}

\begin{section}{Applications of $X/M$.}

We can now work toward answering a question from [19]. In order to motivate the problem, we review a forcing procedure. Suppose we have a supercompact cardinal $\kappa$ and consider the poset $\mathrm{Lv}(\kappa, \omega_1)$ used to L\'{e}vy collapse $\kappa$ to $\omega_2$. By supercompactness of $\kappa$, take an elementary embedding $j: V \to M$ such that $\kappa < j(\kappa)$. We can then transfer a filter $G$ which is $\mathrm{Lv}(\kappa, \omega_1)$ generic over $V$ to get a filter $G^*$ which is $j(\mathrm{Lv}(\kappa, \omega_1))$ generic over $M$ such that $j(p) \in G^*$ whenever $p \in G$. This allows us to extend $j$ to an elementary embedding from $V[G]$ to $M[G^*]$. For an example of this technique, see the proof of Theorem 4.16 in [18].  \\

Observe that there is a homeomorphic copy of $X$ in $M[G^*]$ given by $j"X = \{j(x): x \in X\}$. This is also a subset of $j(X)$, so we can ask whether the subspace topology $\mathcal{S}$ that $j"X$ inherits from $j(X)$ is the same as the topology $\mathcal{T}$ that $j"X$ gets from $X$. This is interesting in view of a longstanding problem due to Hajnal and Juh\'{a}sz [4], which asks whether a Lindel\"{o}f space of size $\aleph_2$ must have a Lindel\"{o}f subspace of size $\aleph_1$. Notice that since $j"U = j"X \cap j(U)$ for an open set $U \subseteq X$, $\mathcal{T}$ is a weaker topology than $\mathcal{S}$. Hence, Tall [19] modified Hajnal and Juh\'{a}sz's question to ask whether a Lindel\"{o}f space of size $\aleph_2$ should have a Lindel\"{o}f weak subspace of size $\aleph_1$. We can now consistently answer this, assuming the existence of a large cardinal. Recall that a space $X$ is projectively countable if $f(X)$ is countable for every continuous map $f$ from $X$ to a separable metric space.

\begin{lemma}\label{lem30} [12] Assume that there is a strongly inaccessible cardinal $\kappa$. Then there is a model obtained by L\'{e}vy collapsing $\kappa$ to $\omega_2$ in which there are no Kurepa trees.\end{lemma}

\begin{theorem}\label{thm31} Assume $CH$. If there is an uncountable Lindel\"{o}f space without a Lindel\"{o}f weak subspace of size $\aleph_1$ then there is a Kurepa tree. \end{theorem}

\begin{proof} Suppose $X$ is an uncountable Lindel\"{o}f space. Assume $CH$ and that there are no Kurepa trees; we will construct a weak subspace of $X$ with size $\aleph_1$. By \ref{lem5}, it suffices to construct $M$ so that $|X/M| = \aleph_1$. Suppose first that $X$ is not projectively countable. Then there is a continuous $f: X \to \mathbb{R}$ with range of size $\aleph_1$. Thus, if $f \in M$, $X/M$ must have size at least $\aleph_1$. For a countable model $M$, $X/M$ will be a Hausdorff space with a countable base, and hence $|X/M| \leq 2^{\aleph_0} = \aleph_1$.  \\

Now, consider a projectively countable $X$, and take a subset $D = \{x_\gamma\}_{\gamma < \omega_1}$ of distinct points. For each pair $(\gamma, \beta) \in \omega_1^2$ we can define a continuous $f_{
\gamma \beta}: X \to \mathbb{R}$ such that $f_{\gamma \beta}(x_\alpha) = 0$ and $f_{\gamma \beta}(x_\beta) = 1$. Enumerate $\bigcup_{(\gamma, \beta) \in \omega_1^2} f_{\gamma \beta}$ as $\{f_\alpha\}_{\alpha< \omega_1}$.\\

Take a countable elementary submodel $M_0$ with $X \in M_0$. We perform a recursive construction. Given $M_\alpha$, let $M_{\alpha+1}$ be the Sk\"{o}lem hull (closure under existential quantification) of $M_\alpha \cup {f_\alpha}$. For a limit ordinal $\lambda < \omega_1$, take $M_\lambda = \bigcup_{\alpha < \lambda} M_\alpha$. This gives us an $\in$-chain of length $\omega_1$, $M_0 \in M_1 \in \cdots \in M_\alpha \in \cdots$. Let $M = \bigcup_{\alpha < \omega_1} M_\alpha$ be the union model. Let $x_\gamma, x_\beta \in D$, then $f_{\gamma \beta} \in M$ so $x_\gamma$ is not $M$-equivalent to $x_\beta$. Thus, $\pi:X \to X/M$ is injective on $D$ and $X/M$ is uncountable. \\

Notice that since each $X/M_\alpha$ is a separable metric space, $X/M_\alpha$ is countable by projective countability of $X$. Hence, $\prod_{\alpha < \omega_1} X/M_{\alpha}$ can be regarded as a tree of height $\omega_1$ with countable levels. That is, if $[x]_\alpha \in X/M_{\alpha}$ and $[y]_\beta \in X/M_\beta$, we say $[x] \leq [y]$ if $[x]_\alpha \subseteq [y]_\beta$, where we view $[x]_\alpha$ and $[y]_\beta$ as subsets of $X$. Consider the map $F: X/M \to \prod_{\alpha < \omega_1} X/M_{\alpha}$ given by sending $[x]$ to $([x]_\alpha)_{\alpha < \omega_1}$, where $[x]_\alpha$ denotes the $M_\alpha$-equivalence class of $x$. Let $[x_0], [x_1] \in X/M$ with $[x_0] \neq [x_1]$. We know that $[x] = \bigcap\{Z: x \in Z \in M, Z$ a zero set$\}$. So, there is a zero set $Z \in M$ such that $x_0 \in Z$ and $x_1 \notin Z$. But we must have $Z \in M_\alpha$ for some $\alpha < \omega_1$. Therefore $[x_0]_\alpha \neq [x_1]_\alpha$ and we have shown that $F$ is an injection from $X/M$ to paths through the tree $\prod_{\alpha < \omega_1} X/M_{\alpha}$. Since $\prod_{\alpha < \omega_1} X/M_{\alpha}$ is not a Kurepa tree, there are only $\aleph_1$ such paths, we must have $|X/M| \leq \aleph_1$. \end{proof}
 
By Lemma \ref{lem30} and Theorem \ref{thm31} we have the following.

\begin{corollary} It is consistent relative to large cardinals that every uncountable Lindel\"{o}f space has a Lindel\"{o}f weak subspace of size $\aleph_1$. \end{corollary}

It is clear from this proof that a Lindel\"{o}f space which is not projectively countable must have a Lindel\"{o}f weak subspace of size $\aleph_1$ under only $CH$. This raises the question of whether our hypothesis that there is a strongly inaccessible cardinal is necessary.

\begin{subsection}{Indestructibly Lindel\"{o}f spaces.}

In [14] Tall introduced the notion of an \textbf{indestructibly Lindel\"{o}f} space, defined to be a Lindel\"{o}f space which remains Lindel\"{o}f in any extension by countably closed forcing. That is, we say a Lindel\"{o}f space $(X, \tau)$ is indestructible if given $G$ generic for a countably closed notion of forcing, $(\check{X}, \tau(G))$ is Lindel\"{o}f. Here, $\tau(G)$ is the topology generated by $\tau$ in the forcing extension. Properties known to imply indestructibility include scatteredness, hereditary Lindel\"{o}fness and having size $\leq \aleph_1$ [14]. Conversely, Max Burke [14] observed that for $\aleph_1 \leq \kappa$, $2^\kappa$ is destructible, demonstrating that even compactness does not imply indestructibility. However, no examples of destructible Lindel\"{o}f spaces with points $G_\delta$ are known. \\

We now employ $X/M$ to give a projective characterization of indestructibility.

\begin{theorem}\label{thm5} Suppose $X$ is a Lindel\"{o}f space and that $X/M$ is indestructibly Lindel\"{o}f for every countably closed $M \prec H_\lambda$ with $X \in M$ and $|M| = \aleph_1$. Then $X$ is indestructibly Lindel\"{o}f. \end{theorem}

\begin{proof} We begin by improving slightly a characterization of indestructibility given in [14]. Recall that a covering tree is a collection of open sets \[\{U_f : f \in \displaystyle{\bigcup_{\alpha < \omega_1} \omega^\alpha} \}\], with the following property. If $f: \alpha \to \omega$, $\alpha < \omega_1$ and we let $f_n$ be $f$ augmented to take the value $f(\alpha) = n$, then $(U_{f_n})_{n<\omega}$ is a cover of $X$ for any $f$. We will say a \textbf{covering cozero tree} is a covering tree consisting of cozero sets. For $f : \alpha \to \omega$, define the $f$-branch $B_f$ of the covering tree to be $(U_{f|\beta})_{\beta < \alpha}$.

\begin{lemma}\label{lem1} A space is indestructible if and only if for each covering cozero tree, the family $(f: B_f \mbox{ covers }X)$ is dense in $\displaystyle{\bigcup_{\alpha < \omega_1} \omega^\alpha}$, partially ordered by reverse extension. That is, $f \leq g$ exactly when $f$ extends $g$. \end{lemma}

\begin{proof}[Proof of Lemma \ref{lem1}]  The implication assuming indestructibility is trivial from Theorem 3 in [14], so suppose that every covering cozero tree has the property stated above. We refer to the proof of Lemma 5.3 in [5]. Suppose that $(X, \mathcal{T})$ is a Lindel\"{o}f space in the ground model such that in a countably closed forcing extension $V[G]$, $(X, \mathcal{T}(G))$ is not Lindel\"{o}f. Let $\mathcal{P}$ be the destroying partial order. Tall observed in the proof of Theorem 3 in [14] that we may assume that the destroying cover $\mathcal{F}$ consists of ground model open sets, since these form a base - by the same reasoning, we may take $\mathcal{F}$ to consist of cozero sets in the ground model. Following Juh\'{a}sz, for $f: \alpha \to \omega_1$ we recursively define conditions $p_f \in P$ and open sets $U_f \in \mathcal{U}$ such that a generic branch in the resulting tree forms a countable cover with no subcover. But this results in a covering cozero tree, so since $\displaystyle{\bigcup_{\alpha < \omega_1} \omega^\alpha}$ is dense in $\mathrm{Fn}(\omega_1, \omega, \omega_1)$ some restriction of the generic branch should be a cover, and we have a contradiction.     \end{proof}

Returning to the proof of \ref{thm5}, consider a covering cozero tree $\mathcal{F}$ for $X$. $\mathcal{F}$ has size $\aleph_1$, thus we may take an elementary submodel $M \prec H_\lambda$ such that $\mathcal{F} \subseteq M$ and $|M| = \aleph_1$. Observe that, writing $\pi$ for the mapping $X \twoheadrightarrow X/M$, $\tilde{F} = \{\pi(U): U \in \mathcal{F}\}$ is a covering tree for $X/M$, where we naturally index $\pi(U_f)$ by $f$. Consider any $f: \alpha \to \omega$, where $\alpha < \omega_1$. By the indestructibility of $X/M$, we know there is a $g: \beta \to \omega$ such that $\alpha < \beta$ and $g|\alpha = f$, with $\bigl( \pi(U_{g|\gamma}) \bigr)_{\gamma < \beta}$ a cover of $X/M$. But then, since each $U_{g|\gamma} \in M$, Lemma \ref{lem9} guarantees that $( U_{g|\gamma} )_{\gamma < \beta} = \bigl(\pi^{-1}\pi(U_{g|\gamma}) \bigr)_{\gamma < \beta}$, so $B_g$ is a branch of $\mathcal{F}$ which covers $X$. Since $g$ extends $f$, the family of such branches is dense, and we have \ref{thm5}. \end{proof}

The following is now immediate from the observation that when $|M| = \aleph_1$, $X/M$ has weight $\aleph_1$.

\begin{theorem}\label{thm4} A space is indestructibly Lindel\"{o}f if all its continuous images in $[0,1]^{\aleph_1}$ are indestructibly Lindel\"{o}f. \end{theorem}

By analogy with projective countability, say a space is \textbf{projectively} $\mathbf{\aleph_1}$ if every continuous image in $[0,1]^{\aleph_1}$ has size $\leq \aleph_1$. We then have

\begin{theorem}\label{thm6} If a Lindel\"{o}f space is projectively $\aleph_1$, then it is indestructibly Lindel\"{o}f \end{theorem}

We can extend this line of reasoning to simplify a question asked by Tall and Scheepers in [11]. First we need some definitions.

\begin{definition} The selection principle $\mathbf{\mathrm{S}^{\omega_1}_1 (\mathcal{O}, \mathcal{O})}$ is the statement that for any sequence $(\mathcal{O}_\alpha)_{\alpha < \omega_1}$ of open covers of a space, there are sets $U_\alpha \in \mathcal{O}_\alpha$ such that $(U_\alpha)_{\alpha < \omega_1}$ is a cover.\end{definition}

\begin{definition} Define a game as follows. In the $\alpha^{th}$ round, \textsc{ONE} chooses an open cover $\mathcal{O}_\alpha$ and \textsc{TWO} chooses a single $U_\alpha \in \mathcal{O}_\alpha$. \textsc{TWO} wins the game if $(U_\alpha)_{\alpha < \omega_1}$ is a cover; otherwise \textsc{ONE} wins. $\mathbf{\mathrm{G}^{\omega_1}_1 (\mathcal{O}, \mathcal{O})}$ is the statement that \textsc{ONE} has no winning strategy in this game.\end{definition}

\begin{problem} [11] In what circumstances does $\mathrm{S}^{\omega_1}_1 (\mathcal{O}, \mathcal{O})$ imply $\mathrm{G}^{\omega_1}_1 (\mathcal{O}, \mathcal{O})$? \end{problem}

\begin{lemma} [11] \label{lem2} A space $X$ is indestructibly Lindel\"{o}f if and only if $\mathrm{G}^{\omega_1}_1 (\mathcal{O}, \mathcal{O})$ holds for $X$. \end{lemma}

It is clear that the property $\mathrm{S}^{\omega_1}_1 (\mathcal{O}, \mathcal{O})$ is preserved under continuous images, so we can combine Theorem \ref{thm5} and Lemma \ref{lem2} to get the following.

\begin{theorem} Suppose $X$ is a Lindel\"{o}f space and $\mathrm{S}^{\omega_1}_1 (\mathcal{O}, \mathcal{O})$ implies $\mathrm{G}^{\omega_1}_1 (\mathcal{O}, \mathcal{O})$ for every continuous image of $X$ in $[0,1]^{\omega_1}$. Then $\mathrm{S}^{\omega_1}_1 (\mathcal{O}, \mathcal{O})$ implies $\mathrm{G}^{\omega_1}_1 (\mathcal{O}, \mathcal{O})$ for $X$. In particular, if $\mathrm{S}^{\omega_1}_1 (\mathcal{O}, \mathcal{O})$ implies $\mathrm{G}^{\omega_1}_1 (\mathcal{O}, \mathcal{O})$ for Lindel\"{o}f spaces of weight $\aleph_1$, then the implication holds for all Lindel\"{o}f spaces. \end{theorem}\end{subsection}\end{section}

\section*{Acknowledgements.}

We thank F.D. Tall for introducing us to this subject, and the anonymous referee for many helpful comments.

\section*{References.}

[1] I. Bandlow. A construction in set-theoretic topology by means of elementary substructures. \emph{Z. Math. Logik Grundlag Math.}, 37(5)  467-480, 1991.\\
\\
$[2]$ A. Dow. Set theory in topology. In M. Husek and J. van Mill, editors, \emph{Recent progress in general topology (Prague 1991)}, pages 167-197. North-Holland, Amsterdam, 1992.\\
\\
$[3]$ T. Eisworth. Elementary submodels and separable monotonically normal compacta. \emph{Top. Proc.} 30(2):431-443,2006.\\
\\
$[4]$ A. Hajnal and I. Juhasz. Remarks on the cardinality of compact spaces and their Lindel\"{o}f subspaces. \emph{Proc. Amer. Math. Soc.}, 59(1):146-148, 1976.\\
\\
$[5]$ I. Juhasz. Cardinal functions. II. In K. Kunen and J.E. Vaughan, editors, \emph{Handbook of Set-theoretic Topology}, pages 63-109. North-Holland, Amsterdam, 1984.\\
\\
$[6]$ L.R. Junqueira. Upwards preservation by elementary submodels. \emph{Top. Proc.} 25:225-249, 2000\\
\\
$[7]$ L.R. Junqueira and F.D. Tall. The topology of elementary submodels. \emph{Top. Appl.} 82:239-266, 1998.\\
\\
$[8]$ L.R. Junqueira and F.D. Tall. More reflections on compactness. \emph{Fund. Math.} 176:127-141, 2003.\\
\\
$[9]$ P. Koszmider and F.D. Tall. A Lindel\"{o}f space with no Lindel\"{o}f subspace of size $\aleph_1$. \emph{Proc. Amer. Math. Soc.}, 130(9):2777-2787, 2002.\\
\\
$[10]$ K. Kunen. \emph{Set theory.} North-Holland, Amsterdam, 1980.\\
\\
$[11]$ M. Scheepers and F.D. Tall. Lindel\"{o}f indestructibility, topological games and selection principles. \emph{Fund. Math.} 210:1-46, 2010.\\
\\
$[12]$ J. Silver. The independence of Kurepa's conjecture and two-cardinal conjectures in model theory. In \emph{Axiomatic Set Theory}, pages 383-390. Amer. Math. Soc., Providence, R.I., 1971.\\
\\
$[13]$ F.D. Tall. Some applications of generalized Martin's Axiom. \emph{Top. Appl.} 57:215-248, 1994.\\
\\
$[14]$ F.D. Tall. On the cardinality of Lindel\"{o}f spaces with points $G_\delta$. \emph{Top. Appl.} 63(1):21-38, 1995.\\
\\
$[15]$ F.D. Tall. Reflecting on compact spaces. \emph{Top. Proc.}, 25:345-350, 2000.\\
\\
$[16]$ F.D. Tall. Reflections on dyadic compacta. \emph{Top. Appl.}, 137:251-258, 2004.\\
\\
$[17]$ F.D. Tall. Compact spaces, elementary submodels and the countable chain condition, II. \emph{Top. Appl.}, 153:273-278, 2006.\\
\\
$[18]$ F.D. Tall. Some problems and techniques in set-theoretic topology. \emph{Set theory and its Applications, Contemp. Math.} ed. L. Babinkostova, A, Caicedo, S. Geschke, M. Scheepers, volume 533, pages 183-209. 2011.\\
\\
$[19]$ F.D. Tall. Set-theoretic problems concerning Lindel\"{o}f spaces. \emph{Q\&A in Gen. Top.}, 29:91-103, 2011.\\
\\
$[20]$ M.D. Weir. Hewitt-Nachbin Spaces. North-Holland, New York, 1975.

\end{document}